\documentclass[11pt,letterpaper,reqno]{amsart}
\usepackage{amssymb}
\usepackage{amsmath}
\usepackage{amsthm}
\usepackage{amsfonts}
\usepackage{bbm}
\usepackage{enumitem} 
\usepackage{booktabs}
\usepackage{graphicx}
\usepackage[T1]{fontenc}
\usepackage{doi}
\usepackage{comment} 
\usepackage{bookmark}
\usepackage{hyperref}
\hypersetup{pdfstartview={FitH}}

\newtheorem{theorem}{Theorem}[section]
\newtheorem{lemma}[theorem]{Lemma}
\newtheorem{proposition}[theorem]{Proposition}

\newtheorem{problem}[theorem]{Problem}

\theoremstyle{definition}

\newtheorem{remark}[theorem]{Remark}

\newcommand{\F}{\mathbb F}
\newcommand{\Z}{\mathbb Z}
\newcommand{\R}{\mathbb R}
\newcommand{\E}{\mathbb E}
\newcommand{\1}{\mathbf 1}
\newcommand{\eps}{\varepsilon}
\DeclareMathOperator{\Aff}{Aff}
\DeclareMathOperator{\Bin}{Bin}

\begin{document}

\title{Almost Affine Invariance over Prime Fields: Green Problem 90}

\author{Jie Ma}
\address{School of Mathematical Sciences, University of Science and Technology of China, Hefei, Anhui 230026, and Yau Mathematical Sciences Center, Tsinghua University, Beijing 100084, China}
\email{jiema@ustc.edu.cn}

\author{Quanyu Tang}
\address{School of Mathematics and Statistics, Xi'an Jiaotong University, Xi'an 710049, P. R. China}
\email{tangquanyu827@gmail.com}

\author{Max Wenqiang Xu}
\address{Courant Institute of Mathematical Sciences, 251 Mercer Street, New York 10012, USA}
\email{maxxu1729@gmail.com}

\begin{abstract}
 Let $A\subset \mathbb{F}_p$ with density 1/2. We call a   set $A$ almost affine invariant under an affine transformation $\phi(x)=ax+b$ if
\[|A \triangle \phi(A)| =o(p).\]
We determine that, the threshold value of $K$ such that $A$ is almost affine invariant simultaneously under all $\phi(x)$ with $|a|, |b|\le K$ and $a\neq 0$, is $K=o(\log p)$. This solves Ben Green's Open Problem 90. 
\end{abstract}

\maketitle

\section{Introduction}

Let $p$ be an odd prime and identify $\Z/p\Z$ with the finite field $\F_p$. 
The following problem is Problem~90 on Ben Green's list of 100 open problems~\cite{GreenProblems}.

\begin{problem}[{\cite[Problem~90]{GreenProblems}}]
\label{prob:green90}
Determine for which ranges of $K=K(p)$ there exists a set
$A\subset \F_p$ of density $1/2$ which is almost invariant under all affine maps
\[
    \phi(x)=ax+b,\qquad 1\le |a|\le K, |b|\le K. \footnote{We note that in the original statement of the problem, it does not exclude the possibility $a\neq 0$ in $\mathbb{F}_p$. But it is clear in this case that $|A \triangle \phi(A)| =(1/2+o(1))p$. So we consider the modified version of the problem.}
\]
\end{problem}
Here \emph{almost invariant} under an affine transformation $\phi$
means that as $p\to \infty$, it holds that
\[
        |A\triangle \phi(A)|=o(p).
\]
Note that a set $A$ is almost invariant under a given $\phi$ implies certain structural information of the set $A$. In other words, the set $A$ needs to satisfy some constraints. The interesting feature of the problem is to investigate how many different constraints that set $A$ can simultaneously satisfy. It is a very natural question to find an optimal upper bound on $K$ for any set $A$ with density 1/2 with the given almost affine invariance property. Here, an optimal bound means that there exists a construction of $A$ such that it does satisfy the property for  $K$ below the bound.

In the comments to Problem~90, Green noted that the problem had been
considered by Eberhard, Mrazovi\'c, and Green in unpublished work \cite{notes}.
He claimed that one can take \(K\to\infty\) slowly, for instance
\(K\ge (\log p)^c\) for some \(c>0\), while \(K\) cannot be as large
as \(p^{1/100}\). He commented that their lower bound is proved by considering the amenability of affine groups, and we also refer to \cite{MOQ,MOA} for related discussions around this idea.

In this paper, we show that $K=o(\log p)$ is the sharp bound and thus determines the threshold.  

\begin{theorem}\label{thm:main}
Let $p$ be a prime and \(K=K(p)\) be a positive integer. If there exist sets
\(A_p\subset \F_p\) with \(|A_p|=(1/2+o(1))p\) such that
\begin{equation}\label{eqn: invariant} 
\max_{\substack{a,b\in\Z\\ 1\le |a|\le K(p),\ |b|\le K(p)\\ p\nmid a}}
        \frac{|A_p\triangle (aA_p+b)|}{p}=o(1),
\end{equation}
then \(K=o(\log p)\). Conversely, if \(K=o(\log p)\), then such
sets \(A_p\) exist.
\end{theorem}

\begin{remark}
    By tracking the proof, one can get a quantitative relation between the two convergence rates $o(1)$ in the statement. We do not pursue it here, and it seems to the authors that the method here would not lead to a sharp dependence estimate. We leave this question to interested readers. 
\end{remark}

We prove the upper bounds in Section~\ref{sec:upper} and give a construction of sets $A_p$ in Section~\ref{sec:lower} satisfying \eqref{eqn: invariant}, as long as $K=o(\log p)$. So the proof of Theorem~\ref{thm:main} is completed by combining the two propositions.

\subsection*{Outline}
We give a quick outline of the proof strategy. For the upper bound proof, we use Fourier analytic tools. Let $f:= 1_A- \frac{|A|}{p}$. 
It is shown that 
translation invariance implies that the Fourier mass
\(\sum_r|\widehat f(r)|^2\) is concentrated on frequencies
\(|r|\ll p/K\).  This is relatively standard and easy to show. Then we restrict ourselves to a probability measure $\mu$ on $\{1, 2, \dots, N\}$ with $N\asymp p/K$ and the assignment to the mass is essentially proportional to the size $|\widehat{f}(r)|^{2}$. 
 To use the multiplicative invariance property, we consider the dilation by primes $q\le K$ and it would imply that $\mu$ is an almost invariant probability
measure on the interval $\{1,\dots,N\}$ under the maps
$n\mapsto qn$.  Moreover, under such a measure, we will show that 
$q$-adic valuation $v_q(n)$ is relatively large and this holds 
for every prime $q\le K$, which is possible only when the total prime weight satisfies that 
$\sum_{q\le K}\log q$ is $o(\log p)$ and this implies that $K=o(\log p)$.

For the lower bound, we use the probabilistic method to construct it. The main part of the proof is to establish a general statement (Lemma~\ref{lem:folner-majority}) about showing the existence of large ``F\o lner sequences'' $A_p$ under actions by elements $s$ in a small family $S_p$ of affine transformations. More precisely, this just means that $|sA_p \triangle A_p|/p$ is small for all $s\in S_p$. The existence is guaranteed under certain conditions, and the key condition is that there exists a small family $\mathcal{F}$ of affine transformations such that $|\mathcal{F} \triangle s^{-1} \mathcal{F}|/|\mathcal{F}|=o(1)$ holds for all $s\in S_p$. Here ``small'' means on the order of $p^{o(1)}$. Lemma~\ref{lem:folner-majority} is proved by the probabilistic method. 
To prove our lower bound, we first need to choose $\mathcal{F}_p$ and $S_p$. And the natural choices are obtained from sets of suitable affine maps over $\mathbb{Q}$ and then modulo $p$ (which leads to affine transformations). Then we get the existence of $A$ by applying Lemma~\ref{lem:folner-majority} to $\mathcal{F}_p$ and $S_p$.   

An interesting step in applying the probabilistic argument is that we first prove that a random $A$ with correct size in expectation, i.e. $\mathbb{E}|A|\sim p/2$, and it satisfies the desired invariant property; however, this can not guarantee the existence of $A$. To do so, we employ the so-called bounded difference inequality to get a concentration estimate which leads to the existence.    

\subsection*{Notation}
Throughout the paper, all $o(1)$ denote the quantities that tend to zero as $p\to\infty$. 
For $r\in\F_p$, write $|r|$ for the absolute value of the representative of $r$ in
$[-(p-1)/2,(p-1)/2]\cap\Z$, and write $e(t)=e^{2\pi i t}$. For a field \(F\), we write
\[
        \Aff(F):=\{\,x\mapsto ax+b:\ a\in F^\times,\ b\in F\,\}
\]
for the affine group over \(F\), with group law given by composition.

\subsection*{Acknowledgements}
We are grateful to Ben Green for detailed comments on an earlier version of the paper, which improved the writing of the work, and for pointing out the relevant discussions on MathOverflow \cite{MOQ, MOA} and references \cite{green2003constructions, Pete2020} and sharing their notes \cite{notes}.

JM was supported by the National Key Research and Development Program of China 2023YFA1010201, National Natural Science Foundation of China grant
12125106, and Innovation Program for Quantum Science and Technology 2021ZD0302902.  
MWX is supported by a Simons Junior Fellowship from the Simons Foundation. 
   
\subsection*{AI Disclosure} We used ChatGPT 5.4 in this work.
AI input plays an important role in some part of the work, although the original arguments came up by AI contain many logical mistakes and gaps along the iterative processes. We provide a summary of the ideas that, in our opinion, are mostly due to AI. First, AI came up with the idea of considering the $q$-adic valuation formulation (see Lemma~\ref{lem:vq}), which provides a clean way in the final step to get the sharp upper bound $o(\log p)$. Second, the idea of using the amenability of the affine group
(especially in an old version of Lemma~\ref{lem:folner-majority}) is due to AI. But we emphasize again that this idea was first used in \cite{notes} (as explicitly mentioned in \cite{GreenProblems}) and also, as Ben Green pointed out to us afterwards, there were related discussions around this idea online (a question asked by Freddie Manners on MathOverflow~\cite{MOQ} and the answer \cite{MOA} by Terry Tao).

\section{The upper bound}\label{sec:upper}

In this section, we prove the upper bound $K=o(\log p)$.

\begin{proposition}\label{prop:upper}
Suppose that $A\subset\F_p$ satisfies $|A|=(1/2+o(1))p$ and
\[
        \max_{\substack{a,b\in\Z\\ 1\le |a|\le K,\ |b|\le K\\ p\nmid a}}
        \frac{|A\triangle(aA+b)|}{p}=o(1).
\]
Then $K=o(\log p)$.
\end{proposition}

\begin{proof}
For any subsequence of primes $p$, 
we may only need to consider the case that $$K = K(p)\to \infty. $$ If not, then $K(p)$ is bounded along the subsequence and the conclusion holds. 

Set
\[
    \eps_p:=
    \max_{\substack{a,b\in\Z\\ 1\le |a|\le K,\ |b|\le K\\ p\nmid a}}
        \frac{|A\triangle(aA+b)|}{p}=o(1),
    \qquad
    \alpha:=\frac{|A|}{p}=\frac12+o(1).
\]
As usual, we do the shift and put $f:=\1_A-\alpha$.  For $r\in\F_p$ define
\[
    \widehat f(r):=\sum_{x\in\F_p} f(x)e(-rx/p).
\]
Since $\widehat f(0)=0$, Parseval identity gives
\[
M:=\sum_{r\in\F_p^\times}|\widehat f(r)|^2 =p\sum_{x\in\F_p}|f(x)|^2
      =\alpha(1-\alpha)p^2
=\left(\frac14+o(1)\right)p^2.
\]
We shall use the probability measure that for $r\in\F_p^\times$
\[
    \mu(r):=\frac{|\widehat f(r)|^2}{M}.
\]
We start to investigate the approximate invariant properties. Our first goal is to show that, by using the approximately translation invariant property, the measure $\mu$ is concentrated on those small $r$, say, no larger than $p/K$. 

For $1\le b\le K$,
\[
    \eps_p p \ge |A\triangle(A+b)|  =  \sum_{x\in\F_p}|f(x+b)-f(x)|^2.
\]
Another application of Parseval identity to function $f(x+b)-f(x)$ yields that for $|b|\le K$, 
\[\sum_{r\in\F_p} |\widehat f(r)|^2|e(br/p)-1|^2\le \eps_p p^2.
\]
Averaging in $1\le b\le K$ gives
\[\sum_{r\in\F_p} |\widehat f(r)|^2 W_K(r/p)\le \eps_p p^2,
    \qquad
    W_K(t):=\frac1K\sum_{b=1}^K |e(bt)-1|^2.
\]

\begin{lemma}\label{lem:WK}
There is an absolute constant $c_0>0$ such that, for every $K\ge2$ and every
$t\in\R/\Z$ with $\|t\|_{\R/\Z}\ge 2/(5K)$, one has $W_K(t)\ge c_0$. Here $\|t\|_{\R/\Z}$ denotes the distance of $t$ to the nearest integer.
\end{lemma}

\begin{proof}
The proof is very straightforward and standard; we omit the proof. 
\end{proof}

Let
\[
        N:=\left\lfloor\frac{p-1}{2K}\right\rfloor,
        \qquad
        I:=\{r\in\F_p^\times: |r|\le N\}.
\]
If $r\notin I$, then $|r|\ge N+1$.  Since
\[
        N+1>\frac{p-1}{2K}\ge \frac{2p}{5K}
\]
for all $p\ge5$, it follows that
\[
        \left\|\frac rp\right\|_{\R/\Z}
        =\frac{|r|}{p}\ge \frac{2}{5K}
\]
for all sufficiently large $p$.  Lemma~\ref{lem:WK} therefore implies that
\(W_K(r/p)\ge c_0\) for every \(r\notin I\), and hence
\[
    c_0\sum_{r\in\F_p^\times\setminus I}|\widehat f(r)|^2\le \eps_p p^2.
\]
Thus
\[
    \sum_{r\in\F_p^\times\setminus I}|\widehat f(r)|^2=o(p^2),
    \qquad\text{and hence}\qquad
    \mu(I)=1-o(1).                         \tag{2.1}\label{eq:mass-in-I}
\]
In particular $I$ is non-empty for all large $p$, so $N\ge1$ and $K\le(p-1)/2$.
All primes $q\le K$ are therefore non-zero modulo $p$.

\bigskip

Next we investigate the dilations $x\mapsto qx$ for primes $q\le K$.  Since
$|A\triangle qA|\le\eps_p p$,
\[
    \sum_{x\in\F_p}|f(q^{-1}x)-f(x)|^2\le \eps_p p.
\]
Taking Fourier transforms, it gives
\[
    \sum_{r\in\F_p}|\widehat f(qr)-\widehat f(r)|^2\le \eps_p p^2.
\]
Combining Cauchy's inequality and the above bound, we have
\[
\begin{aligned}
    \sum_{r\in\F_p^\times}\Big| |\widehat f(r)|^2  -|\widehat f(qr)|^2\Big|
    &\le
    \left(\sum_{r\in\F_p}|\widehat f(qr)-\widehat f(r)|^2\right)^{1/2}
    \left(\sum_{r\in\F_p}(|\widehat f(r)|+|\widehat f(qr)|)^2\right)^{1/2}  \\
    &\ll \sqrt{\eps_p}\,p^2.
\end{aligned}
\]
After division by $M\sim p^2/4$,
\[
    \sum_{r\in\F_p^\times}|\mu(r)-\mu(qr)|=o(1)              \tag{2.2}\label{eq:mu-inv}
\]
uniformly for all primes $q\le K$.  Consequently, for every
$E\subset\F_p^\times$,
\[
        |\mu(E)-\mu(qE)|=o(1)                         \tag{2.3}\label{eq:set-mass}
\]
uniformly for primes $q\le K$.

We first record a rough consequence.  We claim that
\[
        K\le 2N                                                   \tag{2.4}\label{eq:K-le-2N}
\]
for all sufficiently large $p$.  If not, choose a prime $q\le K$ with $q>N$ (when $N=1$ take $q=2$, while for $N\ge2$ this follows from Bertrand's postulate).
Since $qN\le KN\le(p-1)/2$, the set $qI$ does not wrap around modulo $p$, and
$qI\cap I=\varnothing$ (due to $q>N$). This implies that $\mu(qI) +\mu(I)\le 1$. But \eqref{eq:set-mass} gives
$\mu(qI)=\mu(I)+o(1)$. Together, these contradict to \eqref{eq:mass-in-I} which states $\mu(I)=1-o(1)$. This proves
\eqref{eq:K-le-2N}, and we remark that this already gives a bound $K\ll p^{1/2}$.

Set
\[
        M_0:=\min(K,N).
\]
By \eqref{eq:K-le-2N}, $M_0\ge K/2$.  Define a probability measure $\lambda$ on
$\{1,\dots,N\}$ as follows.  For $1\le n\le N$ put
\[
        a_n:=\mu(n)+\mu(-n),
        \qquad
        \lambda(n):=\frac{a_n}{\mu(I)}.
\]
For an integer $q\ge2$, define
\[
    (T_q\lambda)(n):=
    \begin{cases}
        \lambda(qn),& qn\le N,\\
        0,& qn>N.
    \end{cases}
\]

Throughout the next two lemmas, $\|\cdot\|_1$ denotes the unnormalised
$\ell^1$-norm on $\{1,\dots,N\}$.

\begin{lemma}\label{lem:Tq}
By using the approximately dilation invariant property, we have that 
for every prime $q\le M_0$,
\[
        \|\lambda-T_q\lambda\|_{\ell^1(\{1,\dots,N\})}=o(1)
\]
uniformly in $q$.
\end{lemma}

\begin{proof}
Fix a prime $q\le M_0$.  Let
\[
        S_q:=\{\pm n: N/q<n\le N\}\subset\F_p^\times.
\]
Since $q\le K$ and $qN\le KN\le(p-1)/2$, the set $qS_q$ lies outside $I$.
Using \eqref{eq:set-mass} and \eqref{eq:mass-in-I},
\[
        \mu(S_q)\le \mu(qS_q)+o(1)
        \le \mu(\F_p^\times\setminus I)+o(1)=o(1).        \tag{2.5}\label{eq:tail-small}
\]
Moreover, by \eqref{eq:mu-inv},
\[
\begin{aligned}
    \sum_{n\le N/q}|a_n-a_{qn}|
    &\le \sum_{n\le N/q}|\mu(n)-\mu(qn)|
        +\sum_{n\le N/q}|\mu(-n)-\mu(-qn)|  \\
    &\le \sum_{r\in\F_p^\times}|\mu(r)-\mu(qr)|=o(1).
\end{aligned}
\]
Combining this and \(\mu(I)=1-o(1)\), we obtain
\[
    \sum_{n\le N/q}|\lambda(n)-\lambda(qn)|
    \le \frac{1}{\mu(I)}\sum_{n\le N/q}|a_n-a_{qn}|=o(1),
\]
and by using \eqref{eq:tail-small},
\[   \sum_{N/q<n\le N}\lambda(n)
    =\frac{1}{\mu(I)}\sum_{N/q<n\le N}a_n
    =\frac{\mu(S_q)}{\mu(I)}
    =o(1).
\]
Therefore triangle inequality implies that 
\[
    \|\lambda-T_q\lambda\|_1
  \le \sum_{n\le N/q}|\lambda(n)-\lambda(qn)|+
      \sum_{N/q<n\le N}\lambda(n) 
 =o(1).
\qedhere\]
\end{proof}

A quick consequence of the above lemma is that the $q$-adic evaluation $v_q$ along the measure $\lambda$ is large. 
\begin{lemma}\label{lem:vq}
Let $\lambda$ be a probability measure on $\{1,\dots,N\}$ and let $q\ge2$. Let $v_q(n)$ denote the $q$-adic valuation of $n$. If
$\|\lambda-T_q\lambda\|_1\le \eta$ with $0<\eta\le1/8$, then
\[
        \E_\lambda v_q(n)\ge \frac{1}{8\eta}.
\]
\end{lemma}

\begin{proof}
For $t\ge0$ let
\[
    H_t:=\{n\le N:v_q(n)\ge t\},
    \qquad
    R_t:=\{n\le N:v_q(n)=t\}.
\]
Since $(T_q\lambda)(H_t)=\lambda(H_{t+1})$, we have
\[
    \lambda(R_t)=\lambda(H_t)-\lambda(H_{t+1})
            =\lambda(H_t)-(T_q\lambda)(H_t)
            \le \|\lambda-T_q\lambda\|_1\le\eta.
\]
For \(t\ge1\), since $H_t^c=R_0\cup R_1\cup\cdots\cup R_{t-1}$, the bound just proved gives
\[
        \lambda(H_t)\ge 1-\sum_{j=0}^{t-1}\lambda(R_j)\ge 1-t\eta .
\]
If \(m:=\lfloor(2\eta)^{-1}\rfloor\), then for \(1\le t\le m\) we have
\(\lambda(H_t)\ge1/2\), and \(m\ge(4\eta)^{-1}\) because \(\eta\le1/8\).  Hence
\[
    \E_\lambda v_q(n)=\sum_{t\ge1}\lambda(H_t)
    \ge \sum_{t=1}^m \lambda(H_t)
    \ge \frac{m}{2}\ge \frac{1}{8\eta}.
\qedhere\]
\end{proof}

We now finish the proof of Proposition~\ref{prop:upper}.  Since $K\to\infty$ and
$M_0:=\min(K,N)\ge K/2$, we have $M_0\to\infty$.  Put
\[
    \eta_p:=\max\left(p^{-1},\max_{\substack{q\le M_0\\ q\ \mathrm{prime}}}
    \|\lambda-T_q\lambda\|_1\right).
\]
By Lemma~\ref{lem:Tq}, the bound \(\|\lambda-T_q\lambda\|_1=o(1)\) holds uniformly for all primes \(q\le M_0\), and therefore $\eta_p=o(1)$. In particular \(\eta_p\le 1/8\) for all sufficiently large \(p\).  For every prime $q\le M_0$, Lemma~\ref{lem:vq} gives
\[
        \E_\lambda v_q(n)\gg \frac1{\eta_p}.
\]
Since
\[
        \log n=\sum_{\substack{q\le N\\ q\ \mathrm{prime}}}v_q(n)\log q
        \qquad (1\le n\le N),
\]
and \(\log n\le \log N\) for all \(1\le n\le N\), taking expectation with
respect to \(\lambda\) gives
\[
    \log N
    \ge \E_\lambda\log n
    = \sum_{\substack{q\le N\\ q\ \mathrm{prime}}}
      (\log q)\E_\lambda v_q(n)
    \ge \sum_{\substack{q\le M_0\\ q\ \mathrm{prime}}}
      (\log q)\E_\lambda v_q(n)
    \gg \frac{1}{\eta_p}\sum_{\substack{q\le M_0\\ q\ \mathrm{prime}}}\log q.
\]
By using Chebyshev's estimate  $\sum_{\substack{q\le M_0\\ q\ \mathrm{prime}}}\log q\gg M_0$, 
it follows that
$M_0 \ll \eta_p \log N =o(\log p)$. Recall that $ K\le 2M_0$ and this completes the proof. 
\end{proof}

\section{The lower bound}\label{sec:lower}

The goal of this section is to prove the following existence result.
\begin{proposition}\label{prop:lower} Let $p$ be large and 
assume $K=o(\log p)$. Then there is a set $A\subset\F_p$ with\footnote{We remark that one may even show the existence of such a set with $|A|=\lfloor p/2\rfloor$ by further altering $o(p)$ number of elements. We skip the details for the proof of this strengthened claim.} $|A|=(1/2+o(1))p$ such that
\[
    \max_{\substack{a,b\in\Z\\ 1\le |a|\le K,
        \ |b|\le K\\ p\nmid a}}
        \frac{|A\triangle(aA+b)|}{p}=o(1).
\]
\end{proposition}
Our construction is probabilistic in nature. To begin with, we prove the following probabilistic result. 

\begin{lemma}\label{lem:majority}
There is an absolute constant $C$ with the following property.  Let $U,V$ be subsets of a finite index set with $|U|=|V|=n$ being odd, and let
$d:=|U\triangle V|$.  For independent Bernoulli$(\frac{1}{2})$ variables $(\zeta_i)$,
define
\[
M(U):=\1\left[\sum_{i\in U}\zeta_i>\frac n2\right],
    \qquad
    M(V):=\1\left[\sum_{i\in V}\zeta_i>\frac n2\right].
\]
Then
\[
        \mathbb P(M(U)\ne M(V))\le C\left(\frac dn\right)^{1/3}.
\]
\end{lemma}

\begin{proof}
The cases $d=0$ and $d \ge n$ are trivial after enlarging $C$, so we assume
$0<d<n$.  Since $|U|=|V|=n$, the number $d=|U\triangle V|$ is even and
\[
        |U\setminus V|=|V\setminus U|=\frac d2 .
\]
Put $m:=|U\cap V|=n-d/2$.  Let
\[
    Z:=\sum_{i\in U\cap V}\zeta_i,
    \qquad
    X:=\sum_{i\in U\setminus V}\zeta_i,
    \qquad
    Y:=\sum_{i\in V\setminus U}\zeta_i.
\]
Then $Z\sim\Bin(m,1/2)$, $X,Y\sim\Bin(d/2,1/2)$, and these variables are
independent.  Set
\[
        W:=Z+X-\frac n2,
        \qquad
        D:=X-Y.
\]
Since $n$ is odd, $W$ and $W-D$ are non-zero half-integers. If $M(U)\ne M(V)$, then $W$ and $W-D$ have opposite signs, so $|W|\le |D|$.
For any $\gamma>0$,
\[
    \mathbb P(M(U)\ne M(V))
    \le \mathbb P(|W|\le \gamma\sqrt d)
      +\mathbb P(|D|>\gamma\sqrt d).
\]
Now $Z+X\sim\Bin(n,1/2)$, whose largest point mass is $O(n^{-1/2})$.  Therefore
\[
    \mathbb P(|W|\le \gamma\sqrt d)
    \ll \frac{\gamma\sqrt d+1}{\sqrt n}.
\]
Also $\operatorname{Var}(D)=d/4$, so Chebyshev's inequality gives
$\mathbb P(|D|>\gamma\sqrt d)\ll \gamma^{-2}$.  Taking
$\gamma=(n/d)^{1/6}$ yields the result.
\end{proof}

The key step is the following general result on the construction of ``F\o lner sequence'' with respect to the action by a small subset of the affine group $\Aff(F)$. We also refer the readers to \cite{MOQ, MOA}. 
\begin{lemma}\label{lem:folner-majority}
Let $p$ tend to infinity through primes.  For each $p$, let
$S=S_p\subset\Aff(\F_p)$ and let
$\mathcal F=\mathcal F_p\subset\Aff(\F_p)$.  Put
\[ 
        \mathcal H:=\mathcal F\cup\bigcup_{s\in S}s^{-1}\mathcal F \qquad \text{and} \qquad \Delta:=\max_{s\in S}
        \frac{|\mathcal F\triangle s^{-1}\mathcal F|}{|\mathcal F|}.
\]
Assume that $  n:=|\mathcal F|$ is odd and that
\[
        |S|=p^{o(1)},\qquad |\mathcal H|=p^{o(1)},\qquad \Delta=o(1).
\]
Assume also that, for all $h_1,h_2\in\mathcal H$, the affine maps
$h_1^{-1}$ and $-h_2^{-1}$ are distinct, where $-u$ denotes the map
$x\mapsto -u(x)$.  Finally, assume that there is a bijection
$\iota:\mathcal F\to\mathcal F$ such that
\[
        g^{-1}(-x)=-\,\iota(g)^{-1}(x)
        \qquad (g\in\mathcal F,\ x\in\F_p).
\]
Then there exists a set $A\subset\F_p$ such that
\[
        A=-A,\qquad |A|=\left(\frac12+o(1)\right)p \qquad \text{and} \qquad \max_{s\in S}\frac{|A\triangle sA|}{p}=o(1).
\]
\end{lemma}

\begin{proof} The proof is a probabilistic construction. We will first show a set $A$ defined below satisfies all size conditions in expectation, and we then use concentration to pass from expectation to existence.

Let
\[
    \Omega_p:=\{[y]:y\in\F_p\},\qquad [y]:=\{y,-y\},
\]
be the set of orbits of the involution $y\mapsto -y$ on $\F_p$.
For each $\omega\in\Omega_p$, choose an independent
random variable $\xi_\omega$ taking value $\pm1$ with equal probability $1/2$.  Next, we introduce our definition of the random $A$. \footnote{We remark that the basically same construction appeared in the slides \cite{Pete2020} for the work \cite
{HutchcroftPete2020}; also it appeared with similar ideas in \cite[Section 4]{green2003constructions}.} Define
\[
      A:=\left\{x\in\F_p:
        \sum_{g\in\mathcal F}\xi_{[g^{-1}x]}>0\right\}
,\]
which can be viewed as the set consists of elements $x$ such that  $(\xi \ast 1_\mathcal{F}) (x)>0$.
There are no ties because $n$ is odd.

The symmetry assumption implies that $A$ is even (i.e. $x \in A \iff -x\in A$). Indeed,
\[
    \sum_{g\in\mathcal F}\xi_{[g^{-1}(-x)]}
    =\sum_{g\in\mathcal F}\xi_{[-\iota(g)^{-1}(x)]}
    =\sum_{g\in\mathcal F}\xi_{[\iota(g)^{-1}(x)]}
    =\sum_{g\in\mathcal F}\xi_{[g^{-1}(x)]},
\]
since $\iota$ is a bijection of $\mathcal F$.

Let $G\subset\F_p$ be the set of points $x$ such that the map $h\longmapsto [h^{-1}x]$ is injective on $\mathcal H$.  If $x\notin G$, then for some distinct
$h_1,h_2\in\mathcal H$ one has either
$h_1^{-1}x=h_2^{-1}x$ or $h_1^{-1}x=-h_2^{-1}x$.  The first affine equation is
not an identity because $h_1\ne h_2$, and the second is not an identity by the hypothesis that $h_1^{-1}$ and $-h_2^{-1}$ are distinct. Hence each ordered pair contributes at most two points $x\in \F_p\setminus G $, and so
\[
|\F_p\setminus G|\le 2|\mathcal H|^2=o(p).
\]

If $x\in G$, then the orbits $[g^{-1}x]$, $g\in\mathcal F$, are distinct.
Thus $\xi_{[g^{-1}x]}$
are independent random variables for $x\in G$ and this is crucial for us to deduce that 
\[
 \mathbb P(x\in A)=\frac12\qquad (x\in G).
\]
It follows that
\[
        \E|A|=\frac p2+o(p).                         \tag{3.1}\label{eq:transfer-density}
\]

Fix $s\in S$.  For $x\in G$, define
\[
        U_x:=\{[g^{-1}x]:g\in\mathcal F\},
        \qquad
        V_x:=\{[h^{-1}x]:h\in s^{-1}\mathcal F\}.
\]
The injectivity defining $G$ gives $|U_x|=|V_x|=n$ and
\[
        |U_x\triangle V_x|
        =|\mathcal F\triangle s^{-1}\mathcal F|
        \le \Delta n.
\]
Moreover, $\1_A(x)$ and $\1_A(sx)$ are precisely the majority functions
associated with $U_x$ and $V_x$ (note that if $h=s^{-1}g$, then
$h^{-1}x=g^{-1}sx$).  Lemma~\ref{lem:majority} therefore gives
\[
        \mathbb P(\1_A(x)\ne \1_A(sx))\le C\Delta^{1/3}
        \qquad (x\in G).
\]
Since $s$ is a bijection of $\F_p$, we have 
\[
        |A\triangle sA|
        =\sum_{x\in\F_p}\1[\1_A(x)\ne\1_A(sx)].
\]
Thus uniformly for all $s\in S$, it holds that
\[
        \E |A\triangle sA|
        \le C\Delta^{1/3}p+o(p)=o(p).
\]

It remains to pass from expectation to existence. The idea is to show that the size $A$ would not change dramatically as we change small amount of seed random variables $\xi_{\omega}$. Note
that there are $(p+1)/2$
independent seed variables $\xi_{\omega}$ with $\omega\in\Omega_p$.  Changing one seed can affect $\1_A(x)$ only if
$g^{-1}x\in\{y,-y\}$ for some $g\in\mathcal F$, and hence it affects at most
$2n$ values of $\1_A(x)$.  Therefore $|A|/p$ changes by at most $2n/p$; and for any fixed
$s\in S$, the quantity $|A\triangle sA|/p$ changes by at most $4n/p$.

By the bounded-differences inequality
(McDiarmid~\cite[Lemma~(1.2)]{McDiarmid}), applied to the independent seed
variables \((\xi_\omega)_{\omega\in\Omega_p}\), there is an absolute constant
\(c>0\) such that for every \(\eta>0\),
\[
    \mathbb P\left(\left|\frac{|A|}{p}-\E\frac{|A|}{p}\right|>\eta\right)
    \le 2\exp\left(-c\frac{\eta^2p}{n^2}\right),
\]
and, uniformly in \(s\in S\),
\[
    \mathbb P\left(\left|\frac{|A\triangle sA|}{p}
    -\E\frac{|A\triangle sA|}{p}\right|>\eta\right)
    \le 2\exp\left(-c\frac{\eta^2p}{n^2}\right).
\]
Since $n\le|\mathcal H|=p^{o(1)}$, choose
\[
        \eta_p:=\left(\frac{n^2}{p}\right)^{1/4}=o(1).
\]
Then
\[
        \frac{\eta_p^2p}{n^2}
        =\left(\frac{p}{n^2}\right)^{1/2}=p^{1/2-o(1)}.
\]
Since $|S|=p^{o(1)}$, a union bound over $S$ shows that with positive
probability both
\[
        |A|=\left(\frac12+o(1)\right)p \qquad \text{and} \qquad \max_{s\in S}\frac{|A\triangle sA|}{p}=o(1)
\]
hold.  This proves the lemma.
\end{proof}

We now prove Proposition~\ref{prop:lower}, i.e. the condition $K=o(\log p)$ is sharp.

\begin{proof}[Proof of Proposition~\ref{prop:lower}]
We construct $A$ by using Lemma~\ref{lem:folner-majority}. To satisfy the conditions in the lemma, we first construct the desired families $  \widetilde{\mathcal F}$ and $  \widetilde{\mathcal H}$ of rational maps over reals and then modulo $p$. We shall claim those reduced families $  {\mathcal F}, {\mathcal{H}}$ indeed satisfy desired conditions as stated in Lemma~\ref{lem:folner-majority}.

To simplify the writing later, we use some local notations. 
Let
\[
        L:= \log(2K)\sqrt{\log p/K}.
\]
Since $K=o(\log p)$, we have the following relations that we need later
\begin{equation}\label{eq:L-choice}
        L\to\infty,
        \qquad
        \frac{\log(2K)}{L}=\sqrt{K/\log p}\to0,
        \qquad
        \frac{LK}{\log(2K)}=\sqrt{K\log p}=o(\log p).
\end{equation}
In particular $L=o(\log p)$ and $K<p$ for all large $p$.

For every prime $q\le K$, define
\[
        L_q:=\left\lfloor\frac{L}{\log q}\right\rfloor.
\]
If there is no prime $q\le K$, all products below are interpreted as empty
products.  Put
\[
        Q:=\prod_{\substack{q\le K\\ q\ \mathrm{prime}}}q^{L_q},
        \qquad
        \mathcal M:=\left\{
        \prod_{\substack{q\le K\\ q\ \mathrm{prime}}}q^{n_q}:
        -L_q\le n_q\le L_q
        \right\}.
\]
Every $m\in\mathcal M$ is a positive rational number with $Q^{-1}\le m\le Q$,
and $Q^2/m\in\Z$.  Let
\[
        T:=\lfloor LKQ^3\rfloor
\]
and define rational affine maps
\[
        \widetilde g_{m,j}(x):=mx+\frac{mj}{Q^2}
        \qquad (m\in\mathcal M,\ -T\le j\le T).
\]
Set
\[
        \widetilde{\mathcal F}:=\{\widetilde g_{m,j}:m\in\mathcal M,
        -T\le j\le T\}.
\]
Let
\[
        S_+:=\{s_{a,b}(x)=ax+b:1\le a\le K,
        \ |b|\le K\},
\]
and define
\[
        \widetilde{\mathcal H}:=
        \widetilde{\mathcal F}\cup
        \bigcup_{s\in S_+}s^{-1}\widetilde{\mathcal F}.
\]

We record the required size estimates.  We use the standard bound
$\pi(x)\ll x/\log (2 x)$.  Since
\[
    \log Q
    =\sum_{\substack{q\le K\\ q\ \mathrm{prime}}}L_q\log q
    \le L\pi(K)
    \ll \frac{LK}{\log(2K)}=o(\log p),
\]
we have $Q=p^{o(1)}$.  Also, for large $p$,
\[
\begin{aligned}
    \log |\mathcal M|
    &=\sum_{\substack{q\le K\\ q\ \mathrm{prime}}}\log(2L_q+1)  \\
    &\le \pi(K)\log(4L)
      \ll \frac{K\log(4L)}{\log(2K)}=o(\log p),
\end{aligned}
\]
where in the last estimate we recall $L= \log(2K)\sqrt{\log p/K}$.
Finally,
\[
    \log(2T+1)=O(\log L+\log K+\log Q)=o(\log p).
\]
Therefore
\[
        |\widetilde{\mathcal F}|= (2T+1)|\mathcal{M}|= p^{o(1)},
        \qquad
        |\widetilde{\mathcal F}|^2=o(p),
\]
and, since $|S_+|=K(2K+1)=p^{o(1)}$,
\begin{equation}\label{eq:H-size}
        |\widetilde{\mathcal H}|\ll |S_{+}| |\widetilde{\mathcal F}|  = p^{o(1)},
        \qquad
        |\widetilde{\mathcal H}|^2=o(p).
\end{equation}

We next reduce these rational maps modulo $p$.  If $\widetilde h(x)=\alpha x+\beta$ has rational coefficients whose denominators are coprime to $p$, write $\rho_p(\widetilde h)\in\Aff(\F_p)$ for its reduction modulo $p$.

For all sufficiently large $p$, this reduction is defined for every element of
$\widetilde{\mathcal H}$. Indeed, write $m=u/v$ in lowest terms.  Then
$u,v\mid Q$.  The coefficients of $\widetilde g_{m,j}$ have denominators
dividing $vQ^2$, while the coefficients of
$s_{a,b}^{-1}\widetilde g_{m,j}$ have denominators dividing $avQ^2$.
Here $1\le a\le K<p$, and every prime divisor of $Q$ is at most $K$.
Thus all these denominators are coprime to $p$.

The slopes are also nonzero modulo $p$.  Indeed, the slope of
$\widetilde g_{m,j}$ is $m=u/v$, while the slope of
$s_{a,b}^{-1}\widetilde g_{m,j}$ is $m/a=u/(av)$.  Here $u,v\mid Q$ and
$1\le a\le K<p$, and every prime divisor of $Q$ is at most $K$.  Hence
neither the numerator nor the denominator of any such slope is divisible by
$p$.  Thus $\rho_p(\widetilde h)\in\Aff(\F_p)$ for every
$\widetilde h\in\widetilde{\mathcal H}$, and
\[
        \rho_p(\widetilde h^{-1})=\rho_p(\widetilde h)^{-1}.
\]

\begin{lemma}\label{lem:no-collision}
For all sufficiently large $p$, the reduction map $\rho_p$ is injective on
$\widetilde{\mathcal H}$.  Moreover, if
$\widetilde h_1,\widetilde h_2\in\widetilde{\mathcal H}$, then
\[
        \rho_p(\widetilde h_1^{-1})\ne -\rho_p(\widetilde h_2^{-1})
\]
as affine maps $\F_p\to\F_p$, where the minus sign denotes pointwise negation.
\end{lemma}

\begin{proof}
Every $\widetilde h\in\widetilde{\mathcal H}$ is either in
$\widetilde{\mathcal F}$, or has the form $s_{a,b}^{-1}\widetilde g_{m,j}$ with
$1\le a\le K$ and $|b|\le K$.  The first case is included by taking $a=1$ and
$b=0$.  Write $m=u/v$ in lowest terms.  Since $m\in\mathcal M$, we have
$u,v\mid Q$, hence $u,v\le Q$.  A direct computation gives
\[
    \widetilde h^{-1}(x)
    =\widetilde g_{m,j}^{-1}(ax+b)
    =\frac{av}{u}x+\frac{bvQ^2-ju}{uQ^2}.
\]
Thus $\widetilde h^{-1}$ can be written in the form
\[
        \widetilde h^{-1}(x)=\frac{r}{s}x+\frac{c}{d},
\]
where one may take
\[
        r=av,\qquad s=u,\qquad c=bvQ^2-ju,\qquad d=uQ^2 .
\]
In particular $s\mid Q$ and $d\mid Q^3$, so $s$ and $d$ are coprime to $p$.
Moreover
\[
    1\le r\le KQ,
    \qquad
    1\le s\le Q,
    \qquad
    |c|\le 2(L+1)KQ^4,    \qquad
    1\le d\le Q^3.
\]
Let
\[
        D_*:=2(L+1)KQ^4.
\]
Then $D_*$ dominates all the quantities $r,s,d,|c|$ above. 
Moreover, $D_*=p^{o(1)}$.

Suppose first that $\rho_p(\widetilde h_1)=\rho_p(\widetilde h_2)$.  Then also
$\rho_p(\widetilde h_1^{-1})=\rho_p(\widetilde h_2^{-1})$.  Writing
\[
    \widetilde h_i^{-1}(x)=\frac{r_i}{s_i}x+\frac{c_i}{d_i}
    \qquad (i=1,2),
\]
with the above bounds, equality modulo $p$ implies
\[
        r_1s_2\equiv r_2s_1\pmod p,
        \qquad
        c_1d_2\equiv c_2d_1\pmod p.
\]
Moreover, for large $p$, we have
\[
        |r_1s_2-r_2s_1|\le 2D_*^2<p,
        \qquad
        |c_1d_2-c_2d_1|\le 2D_*^2<p.
\]
Thus the two congruences are in fact equalities in $\Z$. Hence
$\widetilde h_1^{-1}=\widetilde h_2^{-1}$, and therefore
$\widetilde h_1=\widetilde h_2$.

Similarly, if
\(\rho_p(\widetilde h_1^{-1})=-\rho_p(\widetilde h_2^{-1})\), then
\[
        r_1s_2\equiv -r_2s_1 \pmod p .
\]
Equivalently, \(p\) divides \(r_1s_2+r_2s_1\).  But \(r_i,s_i\le D_*\), and hence
\[
        0<r_1s_2+r_2s_1\le 2D_*^2<p,
\]
which is impossible.
\end{proof}

From now on identify $\widetilde{\mathcal F}$ and $\widetilde{\mathcal H}$ with
their reductions modulo $p$, and write
\[
        \mathcal F:=\rho_p(\widetilde{\mathcal F}),
        \qquad
        \mathcal H:=\rho_p(\widetilde{\mathcal H}).
\]
For $m\in\mathcal M$ and $|j|\le T$, write
$g_{m,j}:=\rho_p(\widetilde g_{m,j})$.

\begin{lemma}\label{lem:Folner}
For every $s=s_{a,b}\in S_+$,
\[
        \frac{|s^{-1}\mathcal F\triangle\mathcal F|}{|\mathcal F|}
        \le C\left(\frac{\log(2K)}{L}+\frac{KQ^3}{T}\right)=: \delta_p,
\]
where $C$ is absolute and $\delta_p\to0$.
\end{lemma}

\begin{proof}
It is enough to work with the rational model, because reduction modulo $p$ is
injective on $\widetilde{\mathcal F}\cup s^{-1}\widetilde{\mathcal F}\subset
\widetilde{\mathcal H}$.  Since left multiplication by $s$ is a bijection of the
affine group,
\[
    |s^{-1}\widetilde{\mathcal F}\triangle\widetilde{\mathcal F}|
    =|\widetilde{\mathcal F}\triangle s\widetilde{\mathcal F}|.
\]
Thus we only need to bound $|\widetilde{\mathcal F}\setminus s\widetilde{\mathcal F}|$.

For $\widetilde g_{m,j}\in\widetilde{\mathcal F}$,
\[
    s_{a,b}^{-1}\widetilde g_{m,j}
    =\widetilde g_{m/a,\, j-bQ^2/m},
\]
and $bQ^2/m\in\Z$.  Therefore $\widetilde g_{m,j}\notin
s\widetilde{\mathcal F}$ only if either $m/a\notin\mathcal M$, or
$j-bQ^2/m\notin[-T,T]\cap\Z$.

For the slope parameter, write
\[
    a=\prod_{\substack{q\le K\\ q\ \mathrm{prime}}}q^{t_q},
    \qquad t_q=v_q(a).
\]
This is an empty product when \(K=1\), in which case the slope contribution
below is vacuous.  The set \(\mathcal M\) is parameterised by the exponent vector
\[
\Big(\prod_{\substack{q\le K\\ q\ \mathrm{prime}}}
    [-L_q,L_q]\Big)\cap\Z^{\pi(K)} .
\]
For every prime \(q\le K\), we have $ \frac{L}{\log q}\ge \frac{L}{\log(2K)}=\sqrt{\frac{\log p}{K}}\to\infty $. 
Thus, for sufficiently large \(p\), it holds that
\[ L_q=\left\lfloor\frac{L}{\log q}\right\rfloor
    \ge \frac{L}{2\log q}
    \qquad (q\le K,\ q\ \mathrm{prime}).
\]

The map \(m\mapsto m/a\) translates the exponent vector by \((-t_q)_q\).
A union bound over the prime coordinates gives
\[
\begin{aligned}
    \frac{\#\{m\in\mathcal M:m/a\notin\mathcal M\}}{|\mathcal M|}
    &\le \sum_{\substack{q\le K\\ q\ \mathrm{prime}}}
        \frac{t_q}{2L_q+1}  \\
    &\ll \frac1L
        \sum_{\substack{q\le K\\ q\ \mathrm{prime}}} t_q\log q
      = \frac{\log a}{L}
      \le \frac{\log(2K)}{L}.
\end{aligned}
\]

For the translation parameter, fix $m\in\mathcal M$ and set
$u_m:=bQ^2/m\in\Z$.  Since $m\ge Q^{-1}$, $|u_m|\le KQ^3$.  The number of
$j\in[-T,T]\cap\Z$ for which $j-u_m\notin[-T,T]$ is $O(|u_m|)$, and hence the
proportion of such $j$ is $O(KQ^3/T)$.

The above estimates and the fact
$|\widetilde{\mathcal F}\triangle s\widetilde{\mathcal F}|
=2|\widetilde{\mathcal F}\setminus s\widetilde{\mathcal F}|$ imply the stated
bound $\delta_p$. Finally, \(LKQ^3\to\infty\), since \(L\to\infty\), \(K\ge1\), and \(Q\ge1\).
Hence, for all sufficiently large \(p\),
\[
        T=\lfloor LKQ^3\rfloor\ge \frac12 LKQ^3 .
\]
It follows that
\[
        \frac{KQ^3}{T}\ll \frac1L,
\]
and therefore \(\delta_p\to0\) by \eqref{eq:L-choice}.
\end{proof}

Let
\[
        n:=|\mathcal F|=|\mathcal M|(2T+1).
\]
This integer is odd, since each factor $2L_q+1$ and $2T+1$ is odd.

We verify the hypotheses of Lemma~\ref{lem:folner-majority} with
$S=S_+$.  First, $|S_+|=K(2K+1)=p^{o(1)}$, and
\[
        |\mathcal H|=p^{o(1)}
\]
by \eqref{eq:H-size}.  Lemma~\ref{lem:no-collision} gives the required
no-sign-collision condition after reduction modulo $p$, and
Lemma~\ref{lem:Folner} gives
\[
        \max_{s\in S_+}
        \frac{|s^{-1}\mathcal F\triangle\mathcal F|}{|\mathcal F|}
        \le \delta_p=o(1).
\]
It remains only to check the symmetry condition in
Lemma~\ref{lem:folner-majority}.  For $g=g_{m,j}\in\mathcal F$,
\[
        g^{-1}(x)=\frac{x}{m}-\frac{j}{Q^2}.
\]
Hence
\[
        g^{-1}(-x)=-\frac{x}{m}-\frac{j}{Q^2}
        =-\,g_{m,-j}^{-1}(x).
\]
The map $g_{m,j}\mapsto g_{m,-j}$ is a bijection of $\mathcal F$.  Thus the
symmetry hypothesis holds.

Lemma~\ref{lem:folner-majority} therefore gives an even set $A\subset\F_p$ such
that
\[
        |A|=\left(\frac12+o(1)\right)p \qquad \text{and} \qquad \max_{s\in S_+}\frac{|A\triangle sA|}{p}=o(1).
\]
Since $A=-A$, the same estimate also holds for negative dilations.  Indeed, if
$a<0$, put $a'=-a>0$; then
\[
        aA+b=a'(-A)+b=a'A+b.
\]
Therefore
\[
        \max_{\substack{a,b\in\Z\\ 1\le |a|\le K,\ |b|\le K\\ p\nmid a}}
        \frac{|A\triangle(aA+b)|}{p}=o(1),
\]
and the proof is completed. 
\end{proof}


\begin{thebibliography}{9}


\bibitem{notes}
S. Eberhard, B. Green, R. Mrazovi\'c, \emph{Translation and dilation invariance in $\mathbb{Z}/q\mathbb{Z}$},
Unpublished notes.






\bibitem{GreenProblems}
B. Green,
\emph{100 Open Problems},
available at \url{https://people.maths.ox.ac.uk/greenbj/papers/open-problems.pdf}.

\bibitem{green2003constructions}
B.~Green,
\newblock ``Some constructions in the inverse spectral theory of cyclic groups,''
\newblock {\em Combinatorics, Probability and Computing}, {\bf 12}(2) (2003), 127--138.


\bibitem{HutchcroftPete2020}
T.~Hutchcroft and G.~Pete,
\emph{Kazhdan groups have cost 1},
Invent. Math. \textbf{221} (2020), no.~3, 873--891.
\url{https://doi.org/10.1007/s00222-020-00967-6}



\bibitem{MOQ}
Freddie Manners, \emph{Is it known that $(F_p^{\times} \ltimes F_p, F_p)$ is not a relative expander family?}, MathOverflow, URL (version: 2012-03-19): \url{https://mathoverflow.net/q/91657}.



\bibitem{McDiarmid}
C. McDiarmid,
\emph{On the method of bounded differences},
in \emph{Surveys in Combinatorics, 1989} (Norwich, 1989),
London Math. Soc. Lecture Note Ser., vol.~141,
Cambridge Univ. Press, Cambridge, 1989, pp.~148--188.
\doi{10.1017/CBO9781107359949.008}.



\bibitem{Pete2020}
G.~Pete,
\newblock {\em Kazhdan groups have cost 1},
\newblock Presentation slides, 2020.
\newblock Available at: \url{https://math.bme.hu/~gabor/KazhdanTalk.pdf}.










\bibitem{MOA} 
Terry Tao, \emph{Is it known that $(F_p^{\times} \ltimes F_p, F_p)$ is not a relative expander family?}, MathOverflow, URL (version: 2012-03-20): \url{https://mathoverflow.net/q/91675}.




\end{thebibliography}
\end{document}